\documentclass{article}

\usepackage{amssymb,amsmath,color}
\usepackage{amsthm}
\usepackage{url}

\definecolor{red}{rgb}{1,0,0}
\definecolor{blue}{rgb}{.2,.2,.8}

\newtheorem{theorem}{Theorem}[section]
\newtheorem{corollary}[theorem]{Corollary}

\theoremstyle{definition}

\begin{document}

\title{Rank partition functions\\ and truncated theta identities}
\author{Mircea Merca\\
	\footnotesize Department of Mathematics, University of Craiova, 200585 Craiova, Romania\\
	\footnotesize Academy of Romanian Scientists, Ilfov 3, Sector 5, Bucharest, Romania\\
	\footnotesize mircea.merca@profinfo.edu.ro
}
\date{}
\maketitle

\begin{abstract} 
In $1944$, Freeman Dyson defined the concept of rank of an integer partition and introduced without definition the term of crank of an integer partition.
A definition for the crank satisfying the properties hypothesized for it by Dyson was discovered in 1988 by G. E. Andrews and F. G. Garvan.
In this paper, we introduce truncated forms for two theta identities involving the generating functions for partitions with non-negative rank and non-negative crank. As corollaries we derive new infinite families of linear inequalities for the partition function $p(n)$.
The number of Garden of Eden partitions are also considered in this context in order to provide other infinite families of linear inequalities for $p(n)$.
\\ 
\\
{\bf Keywords:}  partitions, Dyson's rank, theta series, inequalities
\\
\\
{\bf MSC 2010:}   05A17, 11P81, 11P83 
\end{abstract}

\section{Introduction}

A partition of a positive integer $n$ is any non-increasing sequence of positive
integers whose sum is $n$ \cite{Andrews76}. Let $p(n)$ denote the number of partitions of $n$ with the usual convention that
$p(0)=1$ and $p(n)=0$ when $n$ is not a non-negative integer. 
Ramanujan proved that for every positive integer $n$, we have:
\begin{align*}
p(5n+4) &\equiv 0 \pmod 5\\
p(7n+5) &\equiv 0 \pmod 7\\
p(11n+6) &\equiv 0 \pmod {11}.
\end{align*}
In order to explain the last two congruences combinatorially, Dyson \cite{Dyson}
introduced the rank of a partition. The rank of a partition is defined to be its
largest part minus the number of its parts. We denote by $N(m,n)$ the number of partitions of $n$ with rank $m$. 
According to Atkin and Swinnerton-Dyer \cite[eq. (2.12)]{Atkin}, the generating function for $N(m,n)$ is given by
\begin{align}\label{GF}
\sum_{n=0}^\infty N(m,n) q^n = \frac{1}{(q;q)_\infty} \sum_{n=1}^\infty (-1)^{n-1} q^{n(3n-1)/2+mn}(1-q^n).
\end{align}
Here and throughout this paper, we use the following customary $q$-series notation:
\begin{align*}
& (a;q)_n = \begin{cases}
1, & \text{for $n=0$,}\\
(1-a)(1-aq)\cdots(1-aq^{n-1}), &\text{for $n>0$;}
\end{cases}\\
& (a;q)_\infty = \lim_{n\to\infty} (a;q)_n;\\
& 		
\begin{bmatrix}
n\\k
\end{bmatrix} 
=
\begin{cases}
\dfrac{(q;q)_n}{(q;q)_k(q;q)_{n-k}}, &  \text{if $0\leqslant k\leqslant n$},\\
0, &\text{otherwise.}
\end{cases}
\end{align*}
We sometimes use the following compressed notations:
\begin{align*}
& (a_1,a_2,\ldots,a_r;q)_n = (a_1;q)_n(a_2,q)_n\cdots(a_r;q)_n,\\
& (a_1,a_2,\ldots,a_r;q)_\infty = (a_1;q)_n(a_2,q)_n\cdots(a_r;q)_\infty.
\end{align*}
Because the infinite product $(a;q)_\infty$ 
diverges when $a\neq 0$ 
and $|q|\geqslant 1$, 
whenever $(a;q)_\infty$ appears in a formula, we shall assume $|q| <1$.

By \eqref{GF}, we immediately deduce that
\begin{equation}\label{eq:1}
\sum_{n=0}^\infty N(n) q^n = \frac{1}{(q;q)_\infty} \sum_{n=0}^\infty (-1)^n q^{n(3n+1)/2} = 1+\sum_{n=1}^\infty q^n 
\begin{bmatrix} 2n-1\\n-1 \end{bmatrix},
\end{equation}
and
\begin{equation}\label{eq:3.1}
\sum_{n=0}^\infty R(n) q^n = \frac{1}{(q;q)_\infty} \sum_{n=1}^\infty (-1)^{n+1} q^{n(3n+1)/2} = \sum_{n=1}^\infty q^{n+1}
\begin{bmatrix}
2n\\n-1
\end{bmatrix},
\end{equation}
where $N(n)$ is the number of partitions of $n$ with non-negative rank and
$R(n)$ is the number of partitions of $n$ with positive rank.
We remark that the sequences $\{N(n)\}_{n>0}$ and $\{R(n)\}_{n>0}$ are known and can be seen in the
On-Line Encyclopedia of Integer Sequence \cite[A064173,A064174]{Sloane}.

Linear inequalities involving Euler’s partition function $p(n)$ have been the subject of recent studies.
In \cite{Andrews12},  Andrews and Merca considered Euler's pentagonal number theorem 
\begin{align*}
\sum_{n=-\infty}^{\infty} (-1)^n q^{n(3n-1)/2} = (q;q)_\infty
\end{align*}
and proved a truncated theorem on partitions.

\begin{theorem}\label{th:1}
	For $k\geqslant 1$,
	\begin{equation*} 
	\frac{1}{(q;q)_\infty} \sum_{j=0}^{k-1} (-1)^j q^{j(3j+1)/2} (1-q^{2j+1}) 
	= 1+(-1)^{k-1} \sum_{n=1}^\infty \frac{q^{{k\choose 2}+(k+1)n}}{(q;q)_n}
	\begin{bmatrix}
	n-1\\k-1
	\end{bmatrix}.
	\end{equation*}	
\end{theorem}

As a consequence of Theorem \ref{th:1}, Andrews and Merca derived the following linear partition inequality:
For $n>0$, $k\geqslant 1$,
\begin{equation}\label{eq:1.1}
(-1)^{k-1} \sum_{j=0}^{k-1} (-1)^j 
\Big( p\big(n-j(3j+1)/2\big) - p\big(n-j(3j+5)/2-1\big) \Big) \geqslant 0, 
\end{equation}
with strict inequality if $n\geqslant k(3k+1)/2$. 

Theorem \ref{th:1} has opened up a new study on truncated theta series and linear partition inequalities.
Other recent investigations involving truncated theta series and linear partition inequalities can be found in several papers by 
Andrews and Merca \cite{Andrews17},
Chan, Ho and Mao \cite{Chan},
Guo and Zeng \cite{Guo}, 
He, Ji and Zang \cite{He},   
Mao \cite{Mao,Mao17}, 
Merca \cite{Merca16}, 
and Merca, Wang and Yee \cite{MWY}.

In this paper, motivated by these results, we shall provide a bisected version of Theorem \ref{th:1}.
The first result contains a truncated form of the identity \eqref{eq:1}.

\begin{theorem}\label{th:2}
	For $|q|<1$ and $k\geqslant 1$, there holds
	\begin{align*} 
	& \frac{1}{(q;q)_\infty} \sum_{j=0}^{k-1} (-1)^j q^{j(3j+1)/2} \\
	&\qquad = 1+\sum_{j=1}^\infty q^j \begin{bmatrix} 2j-1\\j-1 \end{bmatrix}
	+(-1)^{k-1}  \frac{q^{k(3k+1)/2}}{(q,q^3;q^3)_\infty} \sum_{j=0}^\infty \frac{q^{j(3j+3k+2)}}{(q^3;q^3)_j(q^2;q^3)_{k+j}}
	\end{align*}	
	and
	\begin{align*} 
	& \frac{1}{(q;q)_\infty} \sum_{j=0}^{k-1} (-1)^{j} q^{j(3j+5)/2+1}\\
	&\qquad = \sum_{j=1}^\infty q^j \begin{bmatrix} 2j-1\\j-1 \end{bmatrix}
	+(-1)^{k-1}  \frac{q^{k(3k+5)/2+1}}{(q^2,q^3;q^3)_\infty} \sum_{j=0}^\infty \frac{q^{j(3j+3k+4)}}{(q^3;q^3)_j(q;q^3)_{k+j+1}}.
	\end{align*}	
\end{theorem}

An immediate consequence owing to the positivity of the sums on the right hand side of the second identity
is given by the following infinite family of linear partition inequalities.

\begin{corollary}\label{cor:1.4}
	For $n>0$, $k\geqslant 1$,
	$$ (-1)^{k-1} \left( \sum_{j=0}^{k-1} (-1)^{j} p(n-j(3j+5)/2-1)  - N(n) \right) \geqslant 0.$$
	with strict inequality if $n\geqslant k(3k+5)/2+1$.	For example,
	\begin{align*}
	& p(n-1) \geqslant N(n),\\
	& p(n-1)-p(n-5) \leqslant N(n),\\
	& p(n-1)-p(n-5)+p(n-12) \geqslant N(n),\text{ and}\\
	& p(n-1)-p(n-5)+p(n-12)-p(n-22) \leqslant N(n).
	\end{align*}
\end{corollary}

Regarding the inequality \eqref{eq:1.1}, we recall the following
partition theoretic interpretation given by Andrews and Merca \cite[Theorem 1]{Andrews12}:
\begin{equation*}
(-1)^{k-1} \sum_{j=0}^{k-1} (-1)^j 
\Big( p\big(n-j(3j+1)/2\big) - p\big(n-j(3j+5)/2-1\big) \Big)= M_k(n),
\end{equation*}
where $M_k(n)$ is the number of partitions of $n$ in which
$k$ is the least integer that is not a part and there are
more parts $>k$ than there are $<k$. In \cite{Yee} has
given a combinatorial proof of this result. 
We can easily deduce that Corollary \ref{cor:1.4} is equivalent to the following result.

\begin{corollary}\label{cor:1.3}
	For $n>0$, $k\geqslant 1$,
	$$ (-1)^{k-1} \left( \sum_{j=0}^{k-1} (-1)^j p(n-j(3j+1)/2)  - N(n) \right) \geqslant M_k(n),$$
	with strict inequality if  $n\geqslant k(3k+5)/2+1$.
\end{corollary}

The following theorem contains a truncated version of the identity  \eqref{eq:3.1}.
\begin{theorem}\label{th:3.1}
	For $|q|<1$ and $k> 1$, there holds
	\begin{align*} 
	& \frac{1}{(q;q)_\infty} \sum_{j=1}^{k-1} (-1)^{j+1} q^{j(3j+1)/2} \\
	&\qquad = \sum_{j=1}^\infty q^{j+1} \begin{bmatrix} 2j\\j-1 \end{bmatrix}
	+(-1)^{k}  \frac{q^{k(3k+1)/2}}{(q,q^3;q^3)_\infty} \sum_{j=0}^\infty \frac{q^{j(3j+3k+2)}}{(q^3;q^3)_j(q^2;q^3)_{k+j}}
	\end{align*}	
	and
	\begin{align*} 
	& \frac{1}{(q;q)_\infty}\left(1-\sum_{j=0}^{k-1} (-1)^{j} q^{j(3j+5)/2+1}\right)  \\
	&\qquad =1+ \sum_{j=1}^\infty q^{j+1} \begin{bmatrix} 2j\\j-1 \end{bmatrix}
	+(-1)^{k}  \frac{q^{k(3k+5)/2+1}}{(q^2,q^3;q^3)_\infty} \sum_{j=0}^\infty \frac{q^{j(3j+3k+4)}}{(q^3;q^3)_j(q;q^3)_{k+j+1}}.
	\end{align*}	
\end{theorem}

Theorem \ref{th:3.1} is not essentially a new result, it is an equivalent version of Theorem \ref{th:2}.
As a consequence of Theorem \ref{th:3.1} we remark the following equivalent form of Corollary \ref{cor:1.3}.

\begin{corollary}\label{cor:3.2}
	For $n\geqslant 0$, $k>1$,
	$$ (-1)^{k} \left( \sum_{j=1}^{k-1} (-1)^{j+1} p\big(n-j(3j+1)/2\big)  - R(n) \right) \geqslant M_k(n),$$
	with strict inequality if $n\geqslant k(3k+5)/2+1$.
\end{corollary}

Theorems \ref{th:2} and \ref{th:3.1} are good reasons to look for new infinite families of
linear inequalities for the partition function $p(n)$.
The rest of this paper is organized as follows. We will first prove Theorem \ref{th:2} in Section \ref{S2}. In Section \ref{S3}, we consider the partitions with non-negative crank and provide a truncated form of an identity of Auluck \cite{Auluck}. 
Section \ref{S4} is devoted the partitions with rank $-2$ or less. Connections between partitions with rank $-2$ or less and partitions with positive crank are given in this context.

\section{Proof of Theorem \ref{th:2}}
\label{S2}

To prove the theorem, we consider the second identity by Heine's transformation  of ${_{2}}\phi_{1}$ series \cite[(III.2)]{gasper}, namely
\begin{equation}\label{eq:HT}
{_{2}}\phi_{1}\bigg(\begin{matrix}a, b\\ c\end{matrix}\,;q,z\bigg)
= \frac{(c/b,bz;q)_\infty}{(c,z;q)_\infty} 
{_{2}}\phi_{1}\bigg(\begin{matrix}abz/c, b\\ bz\end{matrix}\,;q,c/b\bigg).
\end{equation}
Rewriting \eqref{eq:1} as
\begin{equation*}
\frac{1}{(q;q)_\infty} \sum_{n=0}^{k-1} (-1)^n q^{n(3n+1)/2} 
= 1+\sum_{n=1}^\infty q^n \begin{bmatrix} 2n-1\\n-1 \end{bmatrix}
- \frac{1}{(q;q)_\infty} \sum_{n=k}^{\infty} (-1)^n q^{n(3n+1)/2},
\end{equation*}
we get
\allowdisplaybreaks{
\begin{align*}
& \frac{1}{(q;q)_\infty} \sum_{n=k}^\infty (-1)^n q^{n(3n+1)/2}\\
& \quad = (-1)^k \frac{ q^{k(3k+1)/2}}{(q;q)_\infty} \sum_{n=0}^\infty (-1)^n q^{n(6k+1)/2+3n^2/2}\\
& \quad = (-1)^k \frac{ q^{k(3k+1)/2}}{(q;q)_\infty} \lim_{z\to 0} \sum_{n=0}^\infty \frac{(q^{3k+2}/z;q^3)_n}{(z;q^3)_n} z^n\\
& \quad = (-1)^k \frac{ q^{k(3k+1)/2}}{(q;q)_\infty} \lim_{z\to 0} {_{2}}\phi_{1}\bigg(\begin{matrix}q^{3},  q^{3k+2}/z\\ z\end{matrix}\,;q^{3},z\bigg)\\
& \quad = (-1)^k \frac{ q^{k(3k+1)/2}}{(q;q)_\infty} \\
& \quad\qquad \times\lim_{z\to 0} \frac{(z^2/q^{3k+2},q^{3k+2};q^3)_\infty}{(z;q^3)^2_\infty} 
{_{2}}\phi_{1}\bigg(\begin{matrix}q^{3k+5}/z,  q^{3k+2}/z\\ q^{3k+2}\end{matrix}\,;q^{3},z^2/q^{3k+2}\bigg)\\
& \quad = (-1)^k q^{k(3k+1)/2}\frac{(q^{3k+2};q^3)_\infty}{(q;q)_\infty} \lim_{z\to 0} \sum_{n=0}^\infty \frac{(q^{3k+5}/z,q^{3k+2}/z;q^3)_n}{(q^3,q^{3k+2};q^3)_n} \left(\frac{z^2}{q^{3k+2}} \right)^n \\
& \quad = (-1)^k q^{k(3k+1)/2}\frac{(q^{3k+2};q^3)_\infty}{(q;q)_\infty} \lim_{z\to 0} \sum_{n=0}^\infty \frac{(-1)^n q^{3n(n+1)/2} (q^{3k+2}/z;q^3)_n z^n}{(q^3,q^{3k+2};q^3)_n} \\
& \quad = (-1)^k q^{k(3k+1)/2}\frac{(q^{3k+2};q^3)_\infty}{(q;q)_\infty} \sum_{n=0}^\infty \frac{q^{n(3n+3k+2)}}{(q^3,q^{3k+2};q^3)_n} \\
& \quad = (-1)^k q^{k(3k+1)/2}\frac{(q^{2};q^3)_\infty}{(q;q)_\infty} \sum_{n=0}^\infty \frac{q^{n(3n+3k+2)}}{(q^3;q^3)_n(q^2;q^3)_{n+k}}.
\end{align*}}
The first identity is proved.

Considering Euler's pentagonal number theorem, the identity \eqref{eq:1} becomes
\begin{equation*}
\frac{1}{(q;q)_\infty} \sum_{n=0}^{k-1} (-1)^n q^{n(3n+5)/2+1} = \sum_{n=1}^\infty q^n 
\begin{bmatrix}
2n-1\\n-1
\end{bmatrix}
- \frac{1}{(q;q)_\infty} \sum_{n=k}^\infty (-1)^n q^{n(3n+5)/2+1}.
\end{equation*}
The proof of the second identity is quite similar to the proof of the first one.
We can write
\begin{align*}
& \frac{1}{(q;q)_\infty} \sum_{n=k}^\infty (-1)^n q^{n(3n+5)/2+1}\\
& \quad = (-1)^k \frac{ q^{k(3k+5)/2+1}}{(q;q)_\infty} \sum_{n=0}^\infty (-1)^n q^{n(6k+5)/2+3n^2/2}\\
& \quad = (-1)^k \frac{ q^{k(3k+5)/2+1}}{(q;q)_\infty} \lim_{z\to 0} \sum_{n=0}^\infty \frac{(q^{3k+4}/z;q^3)_n}{(z;q^3)_n} z^n\\
& \quad = (-1)^k \frac{ q^{k(3k+5)/2+1}}{(q;q)_\infty} \lim_{z\to 0} {_{2}}\phi_{1}\bigg(\begin{matrix}q^{3},  q^{3k+4}/z\\ z\end{matrix}\,;q^{3},z\bigg)\\
& \quad = (-1)^k \frac{ q^{k(3k+5)/2+1}}{(q;q)_\infty} \\
& \quad\qquad \times\lim_{z\to 0} \frac{(z^2/q^{3k+4},q^{3k+4};q^3)_\infty}{(z;q^3)^2_\infty} 
{_{2}}\phi_{1}\bigg(\begin{matrix}q^{3k+7}/z,  q^{3k+4}/z\\ q^{3k+4}\end{matrix}\,;q^{3},z^2/q^{3k+4}\bigg)\\
& \quad = (-1)^k q^{k(3k+5)/2+1}\frac{(q^{3k+4};q^3)_\infty}{(q;q)_\infty} \lim_{z\to 0} \sum_{n=0}^\infty \frac{(q^{3k+7}/z,q^{3k+4}/z;q^3)_n}{(q^3,q^{3k+4};q^3)_n} \left(\frac{z^2}{q^{3k+4}} \right)^n \\
& \quad = (-1)^k q^{k(3k+5)/2+1}\frac{(q^{3k+4};q^3)_\infty}{(q;q)_\infty} \lim_{z\to 0} \sum_{n=0}^\infty \frac{(-1)^n q^{3n(n+1)/2} (q^{3k+4}/z;q^3)_n z^n}{(q^3,q^{3k+4};q^3)_n} \\
& \quad = (-1)^k q^{k(3k+5)/2+1}\frac{(q^{3k+4};q^3)_\infty}{(q;q)_\infty} \sum_{n=0}^\infty \frac{q^{n(3n+3k+4)}}{(q^3,q^{3k+4};q^3)_n} \\
& \quad = (-1)^k q^{k(3k+5)/2+1}\frac{(q;q^3)_\infty}{(q;q)_\infty} \sum_{n=0}^\infty \frac{q^{n(3n+3k+4)}}{(q^3;q^3)_n(q;q^3)_{n+k+1}}.
\end{align*}
This concludes the proof.

\section{Truncated identity of Auluck}
\label{S3}

In 1988, Andrews and Garvan \cite{Andrews88} defined the crank of an integer partition as follows. 
The crank of a partition is the largest part of the partition if there are no ones as parts and otherwise is the number of parts larger than the number of ones minus the number of ones. 
If $M(m,n)$ denotes the number of partitions of $n$ with crank $m$, then \cite{Andrews88}:
\begin{align}\label{eq:3.1a}
\sum_{n=0}^\infty \sum_{m=-\infty}^\infty M(m,n) z^m q^n = \frac{(q;q)_\infty}{(zq;q)_\infty (q/z;q)_\infty}.
\end{align}
In this section we denote by $C(n)$ the number of partition of $n$ with non-negative crank.
Recently, Uncu \cite{Uncu} proved that the number of partitions into even number of distinct parts whose odd-indexed parts’ sum is $n$ is
equal to the number of partitions of $n$ with non-negative crank. In this context he provided the following result.

\begin{theorem}\label{Th:Uncu}
	The generating function for partitions with non-negative crank is
	\begin{align*}
	\sum_{n=0}^\infty C(n) q^n = \frac{1}{(q;q)_\infty} \sum_{n=0}^\infty  (-1)^n q^{n(n+1)/2}.
	\end{align*}
\end{theorem}

We remark that this result was proved independently by Ballantine and Merca \cite{BM} in a paper that investigate connections between least $r$-gaps in partitions and partitions with non-negative rank and non-negative crank. In this paper they proved that the number of partitions of $n$ with nonnegative crank is even except when $n$ is twice a generalized pentagonal number.
Very recently,  Andrews and Newman \cite{Andrews19} considered \eqref{eq:3.1a} and provided a different proof for Theorem \ref{Th:Uncu}.
In 2011, Andrews \cite{Andrews11} remarked that the following theta identity
\begin{align}\label{eq:3.2}
\frac{1}{(q;q)_\infty} \sum_{n=0}^\infty  (-1)^n q^{n(n+1)/2}  = \sum_{n=0}^\infty \frac{q^{n(n+1)}}{(q;q)^2_n}
\end{align}
is effectively equivalent to an identity of Auluck \cite[eq. (10)]{Auluck} published in $1951$.
We have the following truncated form of the identity \eqref{eq:3.2}.

\begin{theorem}\label{Th:3.2}
	For $k\geqslant 1$,
	\begin{equation*} 
	\frac{1}{(q;q)_\infty} \sum_{n=0}^{k-1} (-1)^n q^{n(n+1)/2}
	= \sum_{n=0}^\infty \frac{q^{n(n+1)}}{(q;q)^2_n} + (-1)^{k-1} q^{k(k+1)/2} \sum_{n=0}^\infty \frac{q^{n(n+k+1)}}{(q;q)_n(q;q)_{n+k}}.
	\end{equation*}	
\end{theorem}

\begin{proof}
	The proof of this theorem is quite similar to the proof of Theorem \ref{th:2}.
	The identity \eqref{eq:3.2} can be written as:
	$$
	\frac{1}{(q;q)_\infty} \sum_{n=0}^{k-1}  (-1)^n q^{n(n+1)/2} = \sum_{n=0}^\infty \frac{q^{n(n+1)}}{(q;q)^2_n}-\frac{1}{(q;q)_\infty} \sum_{n=k}^\infty  (-1)^n q^{n(n+1)/2}.
	$$	
	We have
	\allowdisplaybreaks{
		\begin{align*}
		& \frac{1}{(q;q)_\infty} \sum_{n=k}^\infty  (-1)^n q^{n(n+1)/2}\\
		&\quad = (-1)^k \frac{q^{k(k+1)/2}}{(q;q)_\infty} \sum_{n=0}^\infty  (-1)^n q^{n(2k+1)/2+n^2/2}\\
		&\quad = (-1)^k \frac{q^{k(k+1)/2}}{(q;q)_\infty} \lim_{z\to 0} \sum_{n=0}^\infty \frac{(q^{k+1}/z;q)_n}{(z;q)_n}\\
		&\quad = (-1)^k \frac{q^{k(k+1)/2}}{(q;q)_\infty} \lim_{z\to 0} \frac{(z^2/q^{k+1},q^{k+1};q)_\infty}{(z;q)^2_\infty} \sum_{n=0}^\infty \frac{(q^{k+2}/z,q^{k+1}/z;q)_n}{(q;q)_n(q^{k+1};q)_n}\left(\frac{z^2}{q^{k+1}} \right)^n 
		\tag{By Heine's transformation \eqref{eq:HT}}\\
		&\quad = (-1)^k q^{k(k+1)/2} \frac{(q^{k+1};q)_\infty}{(q;q)_\infty} \sum_{n=0}^\infty \frac{q^{n(n+k+1)}}{(q,q^{k+1};q)_n}\\
		&\quad = (-1)^k q^{k(k+1)/2} \sum_{n=0}^\infty \frac{q^{n(n+k+1)}}{(q;q)_n(q;q)_{n+k}}.
		\end{align*}
	}
	This concludes the proof.
\end{proof}

In analogy with Corollary \ref{cor:1.3}, we derive a new infinite family of linear inequalities for $p(n)$.

\begin{corollary}
	For $n\geqslant 0$, $k\geqslant 1$,
	$$ (-1)^{k-1} \left( \sum_{j=0}^{k-1} (-1)^{j} p\big(n-j(j+1)/2\big)  - C(n) \right) \geqslant 0,$$
	with strict inequality if $n\geqslant k(k+1)/2$. For example,
	\begin{align*}
	& p(n) \geqslant C(n),\\
	& p(n)-p(n-1) \leqslant C(n),\\
	& p(n)-p(n-1)+p(n-3) \geqslant C(n), \text{ and}\\
	& p(n)-p(n-1)+p(n-3)-p(n-6) \leqslant C(n).	
	\end{align*}
\end{corollary}

\section{Garden of Eden partitions}
\label{S4}

In $2007$, B. Hopkins and J. A. Sellers \cite{Hopkins} provided  a formula that counts the number of partitions of $n$ that have rank $-2$ or less.
Following the terminology of cellular automata and combinatorial game theory, they call these Garden of Eden partitions.
These partitions arise naturally in analyzing the game \textit{Bulgarian solitaire} which was popularized by Gardner \cite{Gardner} in $1983$.
By \eqref{GF}, Hopkins and Sellers obtained
\begin{equation}\label{eq:4.1}
\sum_{n=0}^\infty ge(n) q^n = \frac{1}{(q;q)_\infty} \sum_{n=1}^\infty (-1)^{n-1} q^{3n(n+1)/2},
\end{equation}
where $ge(n)$ counts the Garden of Eden partitions of $n$.  We remark the following theta identity.

\begin{theorem}\label{TH:4.1}
	For $|q|<1$,
	$$ \frac{1}{(q;q)_\infty} \sum_{n=1}^\infty (-1)^{n-1} q^{3n(n+1)/2} 
	= \frac{1}{(q,q^2;q^3)_\infty} \sum_{n=0}^\infty \frac{q^{3(n+1)^2}}{(q^3;q^3)_n(q^3;q^3)_{n+1}}.$$
\end{theorem}

\begin{proof}
	We can write
	\begin{align*}
	& \frac{1}{(q;q)_\infty} \sum_{n=1}^\infty (-1)^{n-1} q^{3n(n+1)/2} \\
	& \quad= \frac{q^3}{(q;q)_\infty} \sum_{n=0}^\infty (-1)^{n} q^{3n^2/2 + 9n/2} \\
	& \quad= \frac{q^3}{(q;q)_\infty} \lim_{z\to 0} \sum_{n=0}^\infty \frac{(q^6/z;q^3)_n}{(z;q^3)_n}z^n\\
	& \quad= \frac{q^3}{(q;q)_\infty} \lim_{z\to 0} \frac{(z^2/q^6,q^6;q^3)_\infty}{(z;q^3)^2_\infty} 
	\sum_{n=0}^\infty \frac{(q^n/z,q^6/z;q^3)_n}{(q^3,q^6;q^3)_n}\left(\frac{z^2}{q^6} \right)^n
	\tag{By Heine's transformation \eqref{eq:HT}}\\  
	& \quad= \frac{q^3(q^6;q^3)_\infty}{(q;q)_\infty} \lim_{z\to 0} \sum_{n=0}^\infty \frac{(-1)^n q^{3n(n+1)/2}(q^6/z;q^3)_n}{(q^3,q^6;q^3)_n}z^n\\ 
	& \quad= \frac{q^3(q^6;q^3)_\infty}{(q;q)_\infty} \sum_{n=0}^\infty \frac{ q^{3n(n+2)}}{(q^3,q^6;q^3)_n}\\
	& \quad= \frac{(q^3;q^3)_\infty}{(q;q)_\infty} \sum_{n=0}^\infty \frac{ q^{3(n+1)^2}}{(q^3;q^3)_n(q^3;q^3)_{n+1}}.
	\end{align*}
\end{proof}

Relating to Theorem \ref{TH:4.1}, we remark that
\begin{align}\label{Eq:9}
\sum_{n=0}^\infty D(n) q^n =\frac{1}{(q;q)_\infty} \sum_{n=1}^\infty (-1)^{n-1} q^{n(n+1)/2} = \sum_{n=0}^\infty \frac{q^{(n+1)^2}}{(q;q)_n(q;q)_{n+1}}
\end{align}
is the generating function for the partitions with positive crank. 
It is an easy exercise to deduce Theorem \ref{TH:4.1} from \eqref{Eq:9} and vice versa.
Connections between Garden of Eden partitions and partitions with positive crank can be easily derived considering Theorem \ref{TH:4.1}.

\begin{corollary}
	For $n\geqslant 0$,
	\begin{align*}
	ge(n) = \sum_{k=0}^{\lfloor n/3 \rfloor} D(j) p_3(n-3j),
	\end{align*}
	where $p_3(n)$ counts partitions of $n$ in which no parts are multiples of $3$.
\end{corollary}	

We have the following truncated form of Theorem \ref{TH:4.1}.

\begin{theorem}\label{TH:4.2}
	For $|q|<1$, $k\geqslant 1$,
	\begin{align*}
	& \frac{1}{(q;q)_\infty} \sum_{n=1}^k (-1)^{n-1} q^{3n(n+1)/2} = \frac{1}{(q,q^2;q^3)_\infty} \sum_{n=0}^\infty \frac{q^{3(n+1)^2}}{(q^3;q^3)_n(q^3;q^3)_{n+1}}\\
	& \qquad\qquad\qquad\qquad + (-1)^{k-1} \frac{q^{3(k+1)(k+2)/2}}{(q,q^2;q^3)_\infty} \sum_{n=0}^\infty \frac{q^{3n(n+k+2)}}{(q^3;q^3)_n(q^{3};q^3)_{n+k+1}}.
	\end{align*}
\end{theorem}

\begin{proof}
	The identity \eqref{eq:4.1} can be written as:
	$$\frac{1}{(q;q)_\infty} \sum_{n=1}^k (-1)^{n-1} q^{3n(n+1)/2} = 
	\sum_{n=0}^\infty ge(n) q^n - \frac{q^3}{(q;q)_\infty} \sum_{n=k}^\infty (-1)^{n} q^{3n(n+3)/2}.$$
	We have
	\allowdisplaybreaks{
		\begin{align*}
		& \frac{q^3}{(q;q)_\infty} \sum_{n=k}^\infty (-1)^{n} q^{3n(n+3)/2}\\
		& \quad = (-1)^k \frac{q^{3(k+1)(k+2)/2}}{(q;q)_\infty} \sum_{n=0}^\infty (-1)^n q^{3n(2k+3)/2+3n^2/2}\\ 
		& \quad = (-1)^k \frac{q^{3(k+1)(k+2)/2}}{(q;q)_\infty} \lim_{z\to 0} \sum_{n=0}^\infty \frac{(q^{3(k+2)/2}/z;q^{3})_n}{(z;q^3)_n}z^n\\
		& \quad = (-1)^k \frac{q^{3(k+1)(k+2)/2}}{(q;q)_\infty} \\
		& \quad\quad \times \lim_{z\to 0} \frac{(z^2/q^{3(k+2)},q^{3(k+2)};q^3)_\infty}{(z;q^3)^2_\infty}
		\sum_{n=0}^\infty \frac{(q^{3(k+3)}/z,q^{3(k+2)}/z;q^3)_n}{(q^3,q^{3(k+2)};q^3)_n} \left( \frac{z^2}{q^{3(k+2)}}\right)^n
		\tag{By Heine's transformation \eqref{eq:HT}}\\ 
		& \quad = (-1)^k q^{3(k+1)(k+2)/2} \frac{(q^{3(k+2)};q^3)_\infty}{(q;q)_\infty}
		\lim_{z\to 0} \sum_{n=0}^\infty (-1)^n \frac{q^{3n(n+1)/2}(q^{3(k+2)}/z;q^3)_n}{(q^3,q^{3(k+2)};q^3)_n} z^n\\
		& \quad = (-1)^k q^{3(k+1)(k+2)/2} \frac{(q^{3(k+2)};q^3)_\infty}{(q;q)_\infty} \sum_{n=0}^\infty \frac{q^{3n(n+k+2)}}{(q^3,q^{3(k+2)};q)_n}\\  
		& \quad = (-1)^k q^{3(k+1)(k+2)/2} \frac{(q^3;q^3)_\infty}{(q;q)_\infty} \sum_{n=0}^\infty \frac{q^{3n(n+k+2)}}{(q^3;q^3)_n(q^{3};q^3)_{n+k+1}}.
		\end{align*}
	}
	The proof follows easily considering Theorem \ref{TH:4.1}.
\end{proof}

On the one hand, as a consequence of Theorem \ref{TH:4.2}, we remark a new infinite family of linear inequalities for the partition function $p(n)$. 

\begin{corollary}\label{Th:4.1}
	For $n\geqslant 0$, $k\geqslant 1$,
	$$ (-1)^{k-1} \left( \sum_{j=1}^{k} (-1)^{j-1} p\big(n-3j(j+1)/2\big)  - ge(n) \right) \geqslant 0,$$
	with strict inequality if $n\geqslant 3(k+1)(k+2)/2$. For example,
	\begin{align*}
	& p(n-3) \geqslant ge(n),\\
	& p(n-3)-p(n-9) \leqslant ge(n),\\
	& p(n-3)-p(n-9)+p(n-18) \geqslant ge(n), \text{ and}\\ 
	& p(n-3)-p(n-9)+p(n-18)-p(n-30) \leqslant ge(n).
	\end{align*}
\end{corollary}

On the other hand, by Theorem \ref{TH:4.2}, we deduce the following truncated version of \eqref{Eq:9}.

\begin{corollary}\label{C:4.5}
	For $|q|<1$, $k\geqslant 1$,
	\begin{align*}
	& \frac{1}{(q;q)_\infty} \sum_{n=1}^k (-1)^{n-1} q^{n(n+1)/2} \\
	& = \sum_{n=0}^\infty \frac{q^{(n+1)^2}}{(q;q)_n(q;q)_{n+1}} + (-1)^{k-1} q^{(k+1)(k+2)/2} \sum_{n=0}^\infty \frac{q^{n(n+k+2)}}{(q;q)_n(q;q)_{n+k+1}}.
	\end{align*}
\end{corollary}

This result allows us to deduce the following infinite family of linear inequalities for the partition function $p(n)$. 

\begin{corollary}\label{C:4.6}
	For $n\geqslant 0$, $k\geqslant 1$,
	$$ (-1)^{k-1} \left( \sum_{j=1}^{k} (-1)^{j-1} p\big(n-j(j+1)/2\big)  - D(n) \right) \geqslant 0,$$
	with strict inequality if $n\geqslant (k+1)(k+2)/2$. For example,
	\begin{align*}
	& p(n-1) \geqslant D(n),\\
	& p(n-1)-p(n-3) \leqslant D(n),\\
	& p(n-1)-p(n-3)+p(n-6) \geqslant D(n), \text{ and}\\ 
	& p(n-1)-p(n-3)+p(n-6)-p(n-10) \leqslant D(n).
	\end{align*}
\end{corollary}

\section{Concluding remarks}

New infinite families of linear inequalities for the partition function $p(n)$ have been introduced in this paper considering
two theta identities involving the generating functions for partitions with non-negative rank and non-negative crank.
Inspired by these results,  in Section \ref{S4} we considered the partitions with rank $\leqslant -2$ (Garden of Eden partitions)
and obtained another infinite families of linear inequalities for $p(n)$.

Theorems \ref{th:1} and \ref{th:2} allow us to derive the following theta identity.

\begin{corollary}\label{C:5.1}
	For $|q|<1$ and $k\geqslant 1$, there holds
	\begin{align*}
	&\sum_{n=1}^\infty \frac{q^{{k\choose 2}+(k+1)n}}{(q;q)_n}
	\begin{bmatrix}
	n-1\\k-1
	\end{bmatrix}\\
	& = \frac{q^{k(3k+1)/2}}{(q,q^3;q^3)_\infty} \sum_{n=0}^\infty \frac{q^{n(3n+3k+2)}}{(q^3;q^3)_n(q^2;q^3)_{n+k}}
	- \frac{q^{k(3k+5)/2+1}}{(q^2,q^3;q^3)_\infty} \sum_{n=0}^\infty \frac{q^{n(3n+3k+4)}}{(q^3;q^3)_n(q;q^3)_{n+k+1}}.
	\end{align*}
\end{corollary}

A similar theta identity can be derived if we consider another truncated form of Euler's pentagonal number theorem given by
D. Shanks \cite{Shanks} in $1951$:
\begin{align}\label{eq:S}
1+\sum_{n=1}^{k} (-1)^n \left(q^{n(3n+1)/2} + q^{n(3n-1)/2} \right) 
= \sum_{n=0}^k (-1)^n \frac{q^{{n+1\choose 2}+kn}(q;q)_k}{(q;q)_n}.
\end{align}

\begin{corollary}\label{C:5.2}
	For $|q|<1$ and $k\geqslant 0$, there holds
	\begin{align*}
	& \frac{(-1)^k}{(q;q)_\infty}\sum_{n=0}^k (-1)^n \frac{q^{{n+1\choose 2}+kn}(q;q)_k}{(q;q)_n}-(-1)^k\\
	& =  \frac{q^{k(3k+7)/2+2}}{(q,q^3;q^3)_\infty} \sum_{n=0}^\infty \frac{q^{n(3n+3k+5)}}{(q^3;q^3)_n(q^2;q^3)_{n+k+1}}
	+ \frac{q^{k(3k+5)/2+1}}{(q^2,q^3;q^3)_\infty} \sum_{n=0}^\infty \frac{q^{n(3n+3k+4)}}{(q^3;q^3)_n(q;q^3)_{n+k+1}}.
	\end{align*}	
\end{corollary}

The Shanks identity \eqref{eq:S} and Corollary \ref{C:5.2} allow us to obtain the following infinite family of linear inequalities: For $n>0$, $k\geqslant 1$,
\begin{equation*}
(-1)^{k}\left( p(n)+ \sum_{j=1}^{k} (-1)^j 
\Big( p\big(n-j(3j+1)/2\big) - p\big(n-j(3j-1)/2\big) \Big) \right) \geqslant 0, 
\end{equation*}
with strict inequality if $n> k(3k+5)/2$. We remark that this inequality is weaker than the inequality \eqref{eq:1.1}. However, a partition theoretic interpretation for it would be very interesting.

Relevant to Theorem \ref{th:2} and Corollaries \ref{C:5.1} and \ref{C:5.2}, it would be very appealing to have combinatorial interpretations for 
$$
\frac{q^{k(3k+1)/2}}{(q,q^3;q^3)_\infty} \sum_{n=0}^\infty \frac{q^{n(3n+3k+2)}}{(q^3;q^3)_n(q^2;q^3)_{n+k}},
$$
$$
\frac{q^{k(3k+5)/2+1}}{(q^2,q^3;q^3)_\infty} \sum_{n=0}^\infty \frac{q^{n(3n+3k+4)}}{(q^3;q^3)_n(q;q^3)_{n+k+1}},
$$
and
$$
\frac{q^{k(3k+7)/2+2}}{(q,q^3;q^3)_\infty} \sum_{n=0}^\infty \frac{q^{n(3n+3k+5)}}{(q^3;q^3)_n(q^2;q^3)_{n+k+1}}.
$$
Finally, with regard to Theorems \ref{Th:3.2} and  \ref{Th:4.1}, partition theoretic interpretation for 
$$
q^{k(k+1)/2} \sum_{n=0}^\infty \frac{q^{n(n+k+1)}}{(q;q)_n(q;q)_{n+k}}
$$
and
$$\frac{q^{3(k+1)(k+2)/2}}{(q,q^2;q^3)_\infty} \sum_{n=0}^\infty \frac{q^{3n(n+k+2)}}{(q^3;q^3)_n(q^{3};q^3)_{n+k+1}}$$
would be very interesting.

\section*{Acknowledgements}
The author expresses his gratitude to Professor George E. Andrews for some helpful suggestions.

\end{document}